\documentclass[11pt,letterpaper]{article}

\usepackage{amsmath,amssymb,amsthm,mathtools}
\usepackage[margin=0.8in,a4paper]{geometry}    
\usepackage{cite} 
\usepackage{caption}
\usepackage{float}
\usepackage{tikz}
\usepackage{authblk}
\usepackage[colorlinks=true,linkcolor=blue,citecolor=blue,urlcolor=blue]{hyperref}
\usepackage{parskip}
\usepackage{enumitem}
 \usepackage{booktabs}
\usepackage{csquotes}
 \usepackage{lmodern}
\usepackage{orcidlink}
\usepackage{xcolor}
\usepackage{tabularx}
 \usetikzlibrary{arrows.meta, positioning, fit}
\usepackage{longtable}

\theoremstyle{plain}
\newtheorem{theorem}{Theorem}
\newtheorem{corollary}{Corollary}
\newtheorem{lemma}{Lemma}
\newtheorem{proposition}{Proposition}

\theoremstyle{definition}
\newtheorem{definition}{Definition}
\newtheorem{example}{Example}

\theoremstyle{remark}
\newtheorem{remark}{Remark}

\author[1]{El-Mehdi Mehiri\thanks{Email: \url{elmehdi.mehiri@emse.fr} }}
\author[1]{Nabil Absi\thanks{Email: \url{absi@emse.fr}}}
\author[1]{Elodie Suzanne\thanks{Email:   \url{elodie.suzanne@emse.fr}}}

\affil[1]{Mines Saint-Etienne, CMP, Department of Manufacturing Sciences and Logistics, F-13120 Gardanne, France.}

\title{\bf  Grid Theory and Polynomiality in Dynamic Lot-Sizing}
\date{\today}

\begin{document}
\maketitle

\begin{abstract}
\noindent Why are some dynamic lot-sizing problems polynomial? We address this question by introducing \emph{Grid Theory}, a structural framework based on cumulative production and the additive structure of production
bounds. For a general single-item dynamic lot-sizing model with lower and upper production bounds, there exists an optimal extreme solution in which, within each regeneration interval, all but at most one production quantity lie on a boundary value. This induces additive grids, and the \textit{Main Grid Theorem} establishes that an optimal
cumulative production trajectory can be restricted to these discrete sets. Although the resulting grids may be exponentially large, we introduce the notion of \emph{additive dimension} to capture production-bound profiles whose boundary sums admit a low-dimensional representation. We show that bounded additive dimension yields a polynomially constructible grid envelope and a polynomial time grid-based dynamic programming algorithm. The framework extends to separable concave costs and establishes polynomial solvability of several families, including constant capacities, minimum order quantities, a fixed number of capacity levels, fixed-degree polynomial capacities, periodic capacities, and piecewise polynomial capacities. In particular, polynomiality may hold even when the number of distinct capacity values grows with the planning horizon. Grid Theory thus identifies additive structure, rather than the number of distinct resource values, as a sufficient mechanism for polynomial solvability.
\end{abstract}

\medskip

\noindent\textbf{Keywords:}
Dynamic lot-sizing; Grid Theory; Additive grid; Additive dimension; Polynomial solvability; Concave costs; Dynamic programming; Production bounds.

\section{Introduction}
\label{sec:introduction}

Dynamic lot-sizing is one of the classical problems of production planning in which structural properties and computational complexity are tightly linked. Since the seminal work of Wagner and Whitin \cite{WagnerWhitin1958}, a central question has been to characterize the production and inventory trajectories that can arise in an optimal solution and to exploit this structure in the design of efficient algorithms. The resulting literature contains both polynomially solvable subclasses and computationally difficult variants. For a comprehensive survey of single-item dynamic lot-sizing models and their extensions, we refer to Brahimi et al.~\cite{BrahimiEtAl2017}.

Production restrictions are a major source of this complexity. Florian, Lenstra, and Rinnooy Kan \cite{FlorianLenstraRinnooyKan1980} established NP-hardness for restricted deterministic production planning problems and contrasted these cases with polynomially solvable capacity structures. Bitran and Yanasse \cite{BitranYanasse1982} subsequently developed a detailed complexity classification for the single-item capacitated lot-sizing problem. At the same time, important structured subclasses are
polynomially solvable. Florian and Klein \cite{FlorianKlein1971} exploited the extreme point structure of capacitated production plans and obtained an $O(T^4)$ algorithm for
the constant-capacity problem with concave costs. Van Hoesel and Wagelmans \cite{vanHoeselWagelmans1996} subsequently obtained an $O(T^3)$ algorithm for constant capacity, concave production costs, and linear holding costs. Under a different structural regime, Chung and Lin \cite{ChungLin1988} derived an $O(T^2)$ algorithm when
setup and unit production costs are nonincreasing and capacities are nondecreasing. 
These results suggest that the presence of capacity constraints alone does not determine tractability. The structure of the data through which these constraints enter the problem is crucial.

A closely related stream concerns lower production bounds and minimum order quantities (MOQs), where a positive production quantity $x_t$ is required to lie between a minimum level $L_t$ and an upper bound $U_t$, with $y_t$ indicating whether production takes place. Variable lower bounds of the  form $L_ty_t\leq x_t$ have been studied from a polyhedral perspective by Constantino \cite{Constantino1998}. For constant minimum order quantity and constant capacity, Okhrin and Richter \cite{OkhrinRichter2011a} developed an $O(T^3)$ algorithm under a restricted cost structure. Hellion, Mangione, and Penz \cite{HellionMangionePenz2012,HellionMangionePenz2013} considered  constant MOQ and capacity under concave production and inventory costs; the corrected
complexity of their algorithm is $O(T^6)$. Park and Klabjan \cite{ParkKlabjan2015} study the more general feasible structure $x_t\in\{0\}\cup[L_t,U_t]$ and establish, among other results, polynomial solvability when the upper bound is constant and the lower bounds and relevant costs are nonincreasing.

The structural foundations of these algorithms are equally important. Florian and Klein \cite{FlorianKlein1971} showed that, within a regeneration structure, an extreme production plan has at most one fractional production period. Swoveland \cite{Swoveland1975} extended this type of result to piecewise concave production and holding-backorder costs, with breakpoints replacing the classical capacity boundaries. For MOQ models, analogous boundary structures appear in Hellion, Mangione, and Penz \cite{HellionMangionePenz2012} and Park and Klabjan \cite{ParkKlabjan2015}. Of particular interest for this paper, Kirca \cite{Kirca1990} exploited the classical extreme point algorithmically by constructing forward and backward sets of cumulative production levels that may occur in an optimal solution and then performing dynamic programming over those levels.


This literature leaves, however, a distinct structural question. A discrete set of candidate cumulative production levels needs to be small. For arbitrary capacity profiles, the number of such levels can grow exponentially with the planning horizon. Conversely, several different resource profiles lead to polynomial algorithms. Existing approaches explain these cases through specialized structural arguments or, more generally, through a fixed collection of special replenishment quantities \cite{LiLi2016}. What is less explicit is a common state-space principle explaining when the cumulative production levels generated by the boundary structure themselves admit a polynomial representation. This motivates the question studied in this paper:

\begin{quote}
\centering
    \textit{Why does the structure of production boundaries induce a polynomially representable cumulative production state space?}
\end{quote}

 
We address this question through what we call \textit{Grid Theory}. The framework separates two issues that are often linked in dynamic programming analyses. The first is an \textit{optimality} issue: which cumulative production levels are sufficient to contain an optimal trajectory? The second is a \textit{compression} issue: when can those levels be represented by only polynomially many states? The first leads to an additive grid theorem; the second leads to a polynomiality theorem based on the additive complexity of the production boundaries.


Our starting point is cumulative production.  In the original lot-sizing formulation, cumulative production are continuous variables and may, a priori, take infinitely many values. Dynamic programming over cumulative production therefore appears to require a continuous state space.

We show that optimality imposes a much stronger structure. Between two consecutive zero-inventory periods, an optimal extreme solution has all but at most one production quantity on a production boundary. In the classical capacitated problem, these boundaries are $0$ and the capacity $C_t$. With a minimum order quantity, lower and upper
production boundaries both appear.

This extreme point property has a striking consequence. Before the possibly interior production period, cumulative production can be constructed by starting from an earlier cumulative demand level and adding production boundaries. After that period, the same cumulative trajectory can be reconstructed backward from a future cumulative demand level by subtracting production boundaries. The interior production quantity itself never needs to be used as a generator.

We call the resulting sets of cumulative production values \emph{additive grids}. Our main structural result, the \emph{Main Grid Theorem}, proves that there exists an optimal solution whose complete cumulative production trajectory belongs to these grids. Thus, optimality reveals a hidden discrete structure inside an apparently continuous optimization problem.

This grid property alone does not imply polynomial solvability. Additive grids may themselves contain exponentially many points. For example, capacities given by distinct powers of two generate exponentially many different subset sums. The relevant question is therefore not only whether an optimal solution belongs to an additive grid, but whether that grid admits a compact representation.

To answer this second question, we introduce the notion of \emph{additive dimension}. Informally, additive dimension measures how many aggregate coordinates are needed to describe the resource quantities that generate the grid.  

This gives the second main result of the paper, the \emph{General Polynomiality Theorem}: bounded additive dimension implies the existence of a polynomially constructible additive grid envelope and, consequently, a polynomial time grid-based dynamic programming algorithm.

This perspective changes the role of dynamic programming in the analysis. The dynamic program is not the source of polynomiality. Once a polynomial additive grid has been identified, the dynamic program is essentially immediate. The substantive question is instead why the continuous cumulative production space can be compressed into a polynomial number of states.

The framework provides a common interpretation of several structured resource profiles, including constant, finite-level, polynomial, periodic, and piecewise polynomial capacities. It also naturally incorporates minimum order quantities. More generally, Grid Theory depends on the additive structure of production boundaries rather than on capacity alone.

This leads to the following interpretation of the contribution. Grid Theory is not primarily an algorithm for the capacitated lot-sizing problem. It is a framework for identifying when the geometry of optimal solutions produces a discrete cumulative state space, and when the additive complexity of that state space is sufficiently small to
yield polynomial solvability.

More specifically, the paper makes the following contributions.
\begin{enumerate}
\item \textbf{Boundary structure.} Building on the classical fractional period structure, we formulate an extreme point boundary results for a single-item dynamic lot-sizing model with arbitrary period-dependent lower and upper production bounds. In every regeneration interval, all but at most one production quantity can be chosen on a production bound.

\item \textbf{Additive cumulative production grids.} We introduce positive and negative additive grids and prove the Main Grid Theorem: there exists an optimal solution whose entire cumulative production trajectory lies on these grids. This extends the cumulative-state construction to arbitrary period-dependent lower and upper boundaries and isolates their additive structure explicitly. 

\item \textbf{Grid-based optimization.} We derive a grid-based dynamic programming algorithm and an equivalent layered shortest path model. Once the grid is sorted, the dynamic program can be implemented in time linear in the grid cardinality per period.

\item \textbf{Additive dimension and polynomiality.} We introduce additive dimension and prove the General Polynomiality Theorem. Unlike the finite set of special quantities frameworks, bounded additive dimension does not require the number of distinct production boundaries to remain fixed. It therefore yields polynomial algorithms for resource profiles with a growing number of distinct values whenever their additive combinations admit a fixed-dimensional representation.

\item \textbf{Unified and broader polynomial families.} We recover several known polynomial families within a common
      additive state framework and derive polynomial algorithms for broader structured capacity profiles. In particular, fixed-degree polynomial and piecewise polynomial capacities may contain a number of distinct capacity values that grows with the planning horizon while still generating polynomial additive grids.
\end{enumerate}

The remainder of the paper is organized as follows.
Section~\ref{sec:structural} develops the structural foundations,
including the cumulative production formulation, regeneration
intervals, and the extreme point boundary theorem.
Section~\ref{sec:grid-theory} introduces positive and negative additive
grids and proves the Main Grid Theorem.
Section~\ref{sec:grid-dp} presents the Grid Dynamic Programming
algorithm, its shortest path interpretation, and its complexity.
Section~\ref{sec:polynomial-grid} develops Polynomial Grid Theory,
introduces additive dimension, and establishes the General
Polynomiality Theorem.
Section~\ref{sec:polynomial-families} applies the theory to several
structured production resource families. Section~\ref{sec:discussion} discusses the scope and limitations of
the framework and identifies directions for its extension. 
Finally, Section~\ref{sec:conclusion} concludes the paper and outlines
directions for further research.

 \section{Structural Foundations}
\label{sec:structural}

This section develops the structural properties underlying the grid representation introduced in Section~\ref{sec:grid-theory}. We consider a single-item dynamic lot-sizing model in which positive production in period $t$ is restricted by a lower
and an upper production bound. The model contains the classical capacitated lot-sizing problem and the minimum-order-quantity variant as special cases.

The main purpose of this section is to recall the extreme point structure that will later explain why cumulative production can be restricted to a discrete set of values.

\subsection{A boundary-constrained dynamic lot-sizing model}
\label{subsec:model}

Let $\mathcal{T}:=\{1,\ldots,T\}$  be a finite planning horizon. For every period $t\in\mathcal{T}$, let $d_t\geq 0$ denote the deterministic demand, $f_t\geq 0$ the setup cost, $p_t\geq 0$ the unit production cost, and $h_t\geq 0$ the unit inventory holding cost.

Positive production in period $t$ is subject to two nonnegative parameters
$L_t$ and $U_t$, satisfying $0\leq L_t\leq U_t$.  If production is activated in period $t$, the production quantity must lie between $L_t$ and $U_t$.

For each $t\in\mathcal{T}$, let $x_t\geq 0$ denote the production quantity, $I_t\geq 0$ the inventory level at the end of the period, and $y_t\in\{0,1\}$ the setup decision. We assume throughout that the instance is feasible. The problem considered is denoted by $\mathcal P$, and is given by 
\begin{align}
\min \quad &
\sum_{t=1}^{T}
\left(
f_t y_t+p_t x_t+h_t I_t
\right)\label{problem:objective}
\\
\text{\rm s.t.}\quad
&
I_t=I_{t-1}+x_t-d_t,
&& t\in\mathcal{T},\label{problem:balance}
\\
&
L_t y_t\leq x_t\leq U_t y_t,
&& t\in\mathcal{T},\label{problem:boundaries}
\\
&
I_t\geq 0,\quad x_t\geq 0,
&& t\in\mathcal{T},\label{problem:continuous}
\\
&
y_t\in\{0,1\},
&& t\in\mathcal{T},\label{problem:binary}
\\
&
I_0=I_T=0.\label{problem:zero-inventory}
\end{align} 
The inventory balance constraints~\eqref{problem:balance}, together with
\(I_t\geq 0\), ensure that demand is satisfied without backlogging. The
setup  constraints~\eqref{problem:boundaries} impose that production can
occur only if a setup is performed and that the production quantity is between the lower bound \(L_t\) and the upper bound \(U_t\). The   constraints~\eqref{problem:zero-inventory} impose zero initial and terminal inventory.

Model~$\mathcal P$ contains several standard lot-sizing structures. In particular, setting $L_t=0$ and $U_t=C_t$, gives the classical single-item capacitated lot-sizing problem with time-dependent capacity $C_t$. A positive value $L_t>0$ introduces a minimum order quantity in period $t$. Thus, both upper and lower production restrictions can be studied within the same framework.

\subsection{Cumulative production formulation}
\label{subsec:cumulative}

For $t\in\mathcal{T}\cup\{0\}$, define the cumulative demand $D_t:=\sum_{i=1}^{t}d_i$ and the cumulative production $P_t:=\sum_{i=1}^{t}x_i$ with $D_0=P_0=0$.

Summing the inventory balance equations~\eqref{problem:balance} from period $1$ to period $t$ gives
\begin{equation}
I_t=P_t-D_t,
\qquad
t\in\mathcal{T}\cup\{0\}.
\label{eq:inventory-cumulative}
\end{equation}
Consequently, the no-backlogging constraints are equivalent to
\begin{equation*}
P_t\geq D_t,
\qquad
t\in\mathcal{T}, 
\end{equation*}
while the terminal condition $I_T=0$ becomes $P_T=D_T$. Moreover,
\begin{equation*}
x_t=P_t-P_{t-1},
\qquad
t\in\mathcal{T}. 
\end{equation*}
so that production bound constraints~\eqref{problem:boundaries} can be written as
\begin{equation*}
L_t y_t
\leq
P_t-P_{t-1}
\leq
U_t y_t,
\qquad
t\in\mathcal{T}. 
\end{equation*}

Thus, once the setup decisions are fixed, production and inventory decisions can be described entirely through the cumulative production trajectory $(P_0,P_1,\ldots,P_T)$. The structural question that will drive the remainder of the paper is therefore: \textit{which values of $P_t$ need to be represented in order to preserve at least one optimal trajectory?}

\subsection{Regeneration intervals}
\label{subsec:regeneration}

A period $t\in\mathcal{T}\cup\{0\}$ is called a \emph{zero-inventory period} if $I_t=0$, or, equivalently by~\eqref{eq:inventory-cumulative}, $P_t=D_t$. Since $I_0=I_T=0$, periods $0$ and $T$ are always zero-inventory periods.

Let $\ell<r$ be two consecutive zero-inventory periods. Then $P_\ell=D_\ell$,   $P_r=D_r$, and $P_t>D_t$ for all $t=\ell+1,\ldots,r-1$. We call $(\ell,r] = \{\ell+1,\ldots,r\}$ a \emph{regeneration interval}.

Within a regeneration interval, the total quantity produced is fixed:
\begin{equation*}
\sum_{t=\ell+1}^{r}x_t
=
D_r-D_\ell. 
\end{equation*}
The production quantities may nevertheless be redistributed among the
periods of the interval, provided that the production bounds and the
strictly positive intermediate inventories remain feasible. This local freedom is precisely what yields the extreme point property recalled in the following.

\subsection{Extreme point structure}
\label{subsec:extreme point}

Fix a setup vector $y\in\{0,1\}^{T}$.  Once $y$ is fixed, the production decisions belong to the bounded polytope 
\begin{equation}
\mathcal{X}(y)
:=
\left\{
x\in\mathbb{R}^{T}: L_t y_t\leq x_t\leq U_t y_t,\ \sum_{u=1}^{t}x_u\geq D_t,\ \sum_{u=1}^{T}x_u=D_T,\ t\in\mathcal{T}\right\}.
\label{eq:fixed-setup-polytope}
\end{equation}
For a fixed setup vector $y$, a production period $t$ is said to be
\emph{partially utilized} if
\begin{equation*}
L_t y_t<x_t<U_t y_t. 
\end{equation*}
Otherwise, $x_t$ lies on one of its production bounds.

The following theorem recalls the fundamental structural property used throughout the paper.

\begin{theorem} 
\label{thm:extreme point}
Let $y\in\{0,1\}^{T}$ and let $x$ be an extreme point of $\mathcal{X}(y)$. Then every regeneration interval associated with $x$ contains at most one partially utilized production period.

In other words, except for at most one period in each regeneration interval, every production quantity belongs to $\mathcal{B}_t:=\{0,L_t,U_t\}$.
\end{theorem}

\begin{proof}
Fix an arbitrary setup vector
$y\in\{0,1\}^{T}$ and an extreme point
$x\in X(y)$.

Consider a regeneration interval $(\ell,r]$
associated with $x$. Suppose, by contradiction,
that two distinct periods $i$ and $j$, with
$\ell<i<j\le r$, are partially utilized. Thus, $L_i y_i < x_i < U_i y_i$ and  $L_j y_j < x_j < U_j y_j$.

For a scalar $\theta$, define the perturbed
production vector $x(\theta)$ by $x_i(\theta)=x_i+\theta$ and $x_j(\theta)=x_j-\theta$,  leaving all other production quantities unchanged.

This perturbation preserves total production
and leaves cumulative production unchanged from
period $j$ onward. For $t=i,\ldots,j-1$,
cumulative production becomes $P_t(\theta)=P_t+\theta$.

Since these periods lie strictly inside the
regeneration interval, we have $P_t>D_t$.
Moreover, both $x_i$ and $x_j$ lie strictly
between their respective production bounds.

Consequently, there exists $\varepsilon>0$
such that both $x(\varepsilon)$ and
$x(-\varepsilon)$ belong to $X(y)$.

These vectors are distinct, and $x=\frac12 x(\varepsilon)
  +\frac12 x(-\varepsilon)$, contradicting the assumption that $x$ is an
extreme point of $X(y)$.

Therefore, every regeneration interval contains
at most one partially utilized production period.

All remaining production quantities lie on their
respective bounds, so that $x_t\in\{0,L_t,U_t\}$. This completes the proof.
\end{proof}

Theorem~\ref{thm:extreme point} shows that the apparent continuous freedom of the production problem is highly structured. Within each regeneration interval, at most one production quantity needs to take a genuinely interior value. All other production quantities can be chosen from the finite boundary sets $\mathcal{B}_t$.

The boundary property does not require any regularity in the sequences $(L_t)_{t=1}^{T}$ and $(U_t)_{t=1}^{T}$. The bounds may be constant, time-dependent, periodic, polynomial, or irregular profiles. No monotonicity, stationarity, or fixed-number-of-levels assumption is needed. Such assumptions will matter only later, when we ask whether the additive state sets generated by these boundaries have polynomial cardinality.
 
Theorem~\ref{thm:extreme point} belongs to a classical line of structural properties for dynamic lot-sizing. In the capacity setting, Florian and Klein \cite{FlorianKlein1971} established that an extreme production sequence contains at most one period whose positive production is strictly below capacity. Swoveland \cite{Swoveland1975} extended this principle to piecewise concave production and holding-backorder costs, using production and inventory breakpoints. For constant lower and upper production bounds, Hellion, Mangione, and Penz \cite{HellionMangionePenz2012} introduced $UL$-capacity constrained sequences, in which production belongs to $\{0,L,U\}$ except in at most one fractional period. Park and Klabjan \cite{ParkKlabjan2015} developed related structures for time-dependent minimum order quantities.

Accordingly, the novelty here is not the fractional-period principle itself. The role of Theorem~\ref{thm:extreme point} is to formulate the required geometric statement for the boundary-constrained model with arbitrary period-dependent $L_t$ and $U_t$, and thereby to identify the finite sets $\mathcal{B}_t=\{0,L_t,U_t\}$ that generate the cumulative production structure studied next.

Section~\ref{sec:grid-theory} translates this local extreme point property into a global statement in cumulative production space. The key observation will be that, on either side of the unique possible partially utilized period of a regeneration interval, cumulative production is obtained from a zero-inventory by adding or subtracting elements of the boundary sets $\mathcal{B}_t$.

\section{Grid Theory}
\label{sec:grid-theory}

The extreme point structure established in Theorem~\ref{thm:extreme point} has an important consequence in cumulative production space. Within every regeneration interval, all but at most one production quantity belong to the finite boundary sets $\mathcal{B}_t$.  As a result, cumulative production can be represented, on either side
of the possibly partially utilized period, by starting from a cumulative demand level and adding or subtracting production boundaries.

This section formalizes that observation. We first define the additive sets generated by the period specific boundaries and use them to construct forward and backward cumulative production grids. We then prove the Main Grid Theorem, which establishes that these grids contain the complete cumulative production trajectory of at least one optimal solution. Finally, we show that the optimality-preserving grid may have an exponential cardinality. This distinction between \textit{grid existence} and \textit{grid compressibility} is the key step toward the polynomiality theory developed later.

\subsection{Boundary sums and additive grids}
\label{subsec:additive grids}

For a given interval of periods $J\subseteq\mathcal{T}$, define the associated \emph{boundary sum set} by
\begin{equation}
\mathcal{A}(J)
:=
\left\{
\sum_{u\in J} b_u
:
b_u\in\mathcal{B}_u
\text{ for every }u\in J
\right\}.
\label{eq:boundary-sum-set}
\end{equation}
For the empty interval, we adopt the convention $\mathcal{A}(\varnothing):=\{0\}$.  

Thus, $\mathcal{A}(J)$ contains all quantities that can be obtained by selecting, independently for every period $u\in J$, one of the boundary values $0$,  $L_u$, $U_u$, and adding the resulting values.

The relevance of ~\eqref{eq:boundary-sum-set} follows from Theorem~\ref{thm:extreme point}. If a portion of a regeneration interval does not contain its partially utilized production period, then every production quantity in that portion belongs to its
corresponding boundary set. Hence, the change in cumulative production over that portion belongs to an appropriate boundary sum set.

The cumulative demand levels $D_0,D_1,\ldots,D_T$ play the role of \emph{anchors}. Starting from one such anchor and adding or subtracting a boundary sum generates a candidate cumulative production level.

This motivates the following two constructions.

\subsection{Positive and negative grids}
\label{subsec:positive-negative-grids}

Consider a period $t\in\mathcal{T}\cup\{0\}$. For every earlier cumulative demand anchor $D_j$, with $0\leq j\leq t$, production between periods $j+1$ and $t$ may generate cumulative production levels of the form $D_j+a$ with $a\in\mathcal{A}(\{j+1,\ldots,t\})$. We therefore define the \emph{positive grid} at period $t$ by
\begin{equation}
G_t^{+}
:=
\bigcup_{j=0}^{t}
\left(
D_j+
\mathcal{A}(\{j+1,\ldots,t\})
\right),
\label{eq:positive-grid}
\end{equation}
where, for a scalar $D_j$ and a set $A$, $D_j+A:=\{D_j+a:a\in A\}$.

Equivalently,
\begin{equation}
G_t^{+}
=
\left\{
D_j+
\sum_{u=j+1}^{t}b_u
:
0\leq j\leq t,\;
b_u\in\mathcal{B}_u
\right\}. 
\label{eq:positive-grid-expanded}
\end{equation}
A point in $G_t^{+}$ therefore admits a \emph{forward representation}: cumulative production is anchored at an earlier cumulative-demand level $D_j$ and reconstructed by adding production-boundary quantities from periods $j+1,\ldots,t$.


Symmetrically, define the \emph{negative grid} at period $t$ by:


\begin{equation}
G_t^{-}
:=
\bigcup_{j=t}^{T}
\left(
D_j-
\mathcal{A}(\{t+1,\ldots,j\})
\right), 
\label{eq:negative-grid}
\end{equation}
where $D_j-A:=\{D_j-a:a\in A\}$.

Equivalently,
\begin{equation}
G_t^{-}
=
\left\{
D_j-
\sum_{u=t+1}^{j}b_u
:
t\leq j\leq T,\;
b_u\in\mathcal{B}_u
\right\}. 
\label{eq:negative-grid-expanded}
\end{equation} 
A point in $G_t^{-}$ thus admits a \emph{backward representation}: cumulative production is anchored at a future cumulative-demand level $D_j$ and reconstructed by subtracting production-boundary quantities from periods $t+1,\ldots,j$.

The two constructions are combined in the \emph{period-$t$ additive grid} by
\begin{equation}
G_t
:=
\left(
G_t^{+}\cup G_t^{-}
\right)
\cap [D_t,D_T].
\label{eq:period-grid}
\end{equation}

The interval restriction retains only values that can be cumulative production states of a feasible solution at period $t$, since every feasible trajectory satisfies $D_t\leq P_t\leq D_T$.
 
We also define the \emph{global additive grid}
\begin{equation}
G
:=
\bigcup_{t=0}^{T}G_t.
\label{eq:global-grid}
\end{equation}
The period-dependent grids preserve more temporal information than the global grid. In particular, $G_t^{+}$ only uses production boundaries from periods preceding $t$, whereas $G_t^{-}$ only uses production boundaries from periods following $t$. This temporal structure will be useful both in the proof below and in the dynamic program of Section \ref{sec:grid-dp}.

\subsection{The Main Grid Theorem}
\label{subsec:main-grid-theorem}

We can now state the central structural result, which we call, the \textit{Main Grid Theorem} (MGT). 

\begin{theorem} 
\label{thm:main-grid}
There exists an optimal solution of problem~$\mathcal P$ whose cumulative production trajectory satisfies $P_t\in G_t$ for all $t\in\mathcal{T}\cup\{0\}$.  
\end{theorem}

\begin{proof}
By Theorem~\ref{thm:extreme point}, there exists an optimal solution such that every regeneration interval contains at most one partially utilized production period.

Consider such an optimal solution and an arbitrary regeneration interval $(\ell,r]$. By definition, $P_\ell=D_\ell$ and $P_r=D_r$.  Suppose first that the regeneration interval contains a partially utilized period $k$.

Consider a period $t$ satisfying $\ell\leq t<k$.  Since the unique possible partially utilized period is $k$, every production quantity in $\{\ell+1,\ldots,t\}$ belongs to its corresponding boundary set. Hence, $x_u\in\mathcal{B}_u$ for all $u=\ell+1,\ldots,t$. 
 
Using the left endpoint of the regeneration interval gives
\begin{equation*}
P_t
=
P_\ell+
\sum_{u=\ell+1}^{t}x_u
=
D_\ell+
\sum_{u=\ell+1}^{t}x_u. 
\end{equation*}
Therefore, $\sum_{u=\ell+1}^{t}x_u
\in
\mathcal{A}(\{\ell+1,\ldots,t\})$,  and consequently $P_t\in G_t^{+}$.

Now consider a period $t$ satisfying $k\leq t\leq r$. Using the right endpoint of the regeneration interval,
\begin{equation*}
P_t
=
P_r-
\sum_{u=t+1}^{r}x_u
=
D_r-
\sum_{u=t+1}^{r}x_u. 
\end{equation*}
Since $t\geq k$, the possibly partially utilized period $k$ does not appear in the sum on the right-hand side. Hence, $x_u\in\mathcal{B}_u$ for all $u=t+1,\ldots,r$,  so that $\sum_{u=t+1}^{r}x_u
\in
\mathcal{A}(\{t+1,\ldots,r\})$. It follows that $P_t\in G_t^{-}$. 
 
Suppose now that the regeneration interval contains no partially utilized production period. Then every production quantity in the interval belongs to its corresponding boundary set. For every $t\in\{\ell,\ldots,r\}$,
\begin{equation*}
P_t
=
D_\ell+
\sum_{u=\ell+1}^{t}x_u, 
\end{equation*}
and therefore $P_t\in G_t^{+}$. 
 
Thus, in all cases, $P_t\in G_t^{+}\cup G_t^{-}$. Moreover,  feasibility additionally implies $D_t\leq P_t\leq D_T$.  Hence,
\begin{equation*}
P_t\in
\left(
G_t^{+}\cup G_t^{-}
\right)
\cap[D_t,D_T]
=
G_t.
\end{equation*}

Since the argument applies to every regeneration interval, we obtain $P_t\in G_t$ for all $t\in\mathcal{T}\cup\{0\}$. 
\end{proof}

Since $G_t\subseteq G$ for every $t$,  we immediately obtain the following consequence.

\begin{corollary}
There exists an optimal solution of problem~$\mathcal P$ satisfying $P_t\in G$ for all $t\in\mathcal{T}\cup\{0\}$. 
\end{corollary}
 
Theorem~\ref{thm:main-grid}, or the MGT,  converts the structural statement of
Theorem~\ref{thm:extreme point} into a global statement about optimal cumulative production trajectories. Although the production variables are continuous, at least one optimal trajectory is supported entirely on the finite sets $G_0,G_1,\ldots,G_T$.

The cumulative production point of view has an important antecedent in K{\i}rca \cite{Kirca1990}. For the capacitated single-item problem with concave costs, K{\i}rca constructs forward and backward sets of cumulative production levels from regeneration points, takes their union, and applies dynamic programming over the resulting state sets. The present construction should not be interpreted as introducing cumulative production as an optimality-preserving discrete state description.

The role of the Main Grid Theorem is more specific. First, it formulates the forward/backward representation for arbitrary period-dependent lower and upper production boundaries $\mathcal{B}_t=\{0,L_t,U_t\}$,
rather than only the classical capacity boundaries $\{0,C_t\}$. Second, and more importantly for the polynomiality analysis, equations \eqref{eq:positive-grid}-\eqref{eq:negative-grid} expose these states explicitly as cumulative-demand anchors plus or minus \emph{additive combinations of boundary values}. This representation identifies the algebraic object whose cardinality must be controlled if grid-based dynamic programming is to be polynomial.


\subsection{Interpretation: optimality versus state-space compression}
\label{subsec:grid-interpretation}

The Main Grid Theorem separates two questions that are easy to confuse.

The first is an \textit{optimality} question:

\begin{center}
\textit{Which cumulative production states are sufficient to contain at least one optimal trajectory?
}\end{center}

The MGT answers this question: the additive grids $G_t$ are sufficient.

The second is a \textit{complexity} question:

\begin{center}
\textit{How many distinct states do these grids contain, or how compactly can they be represented?}
\end{center}
The MGT alone does not answer this question.

The significance of  the MGT goes beyond the construction of a particular dynamic programming state space.



The structural mechanism behind the first answer is simple but important. A regeneration interval provides two cumulative-demand anchors, $P_\ell=D_\ell$, $P_r=D_r$.

Except for at most one period, production is selected from the boundary sets $\mathcal{B}_t$. Consequently, before the exceptional period cumulative production can be written in the forward form
\begin{equation*}
\begin{array}{ccc}
D_\ell
&
\xrightarrow{\quad +\;\text{boundaries}\quad}
&
P_t,
\end{array} 
\end{equation*}
\begin{equation*}
P_t = D_\ell + \sum_{u=\ell+1}^{t}b_u,\;b_u\in\mathcal{B}_u,
\end{equation*}

whereas at and after the exceptional period it can be written in the backward form

\begin{equation*}
\begin{array}{ccc}
P_t
&
\xleftarrow{\quad -\;\text{boundaries}\quad}
&
D_r.
\end{array} 
\end{equation*}
\begin{equation*}
P_t = D_r - \sum_{u=t+1}^{r}b_u, \; b_u\in\mathcal{B}_u.
\end{equation*}

The partially utilized production quantity is never required as a grid generator. It is bypassed by changing the direction from which cumulative production is represented. This is precisely why both positive and negative grids are needed.

This observation is central to Grid Theory. The relevant discrete structure is not imposed externally as a discretization of a continuous problem. Rather, it is generated endogenously by the regeneration anchors and the boundary structure of optimal extreme solutions.

In the classical capacitated lot-sizing problem, $L_t=0$ and $U_t=C_t$, so $\mathcal{B}_t=\{0,C_t\}$. The grid is then generated by capacity values. With an MOQ, both $L_t$ and $U_t$ become nonzero grid generators. The primitive object in the present work is therefore not "capacity" itself but the collection of period-specific boundary sets $\mathcal{B}_1,\ldots,\mathcal{B}_T$. 
 
This distinction becomes essential when studying polynomiality. because the relevant combinatorial object is \textit{the family of additive sums generated by these boundaries}.

\subsection{The additive grid can be exponential}
\label{subsec:exponential-grid}

Finiteness of the additive grid does not imply polynomial cardinality. The phenomenon already occurs in the classical capacitated lot-sizing problem.

\begin{example}
\label{ex:exponential-grid}

Consider the special case $L_t=0$ and $U_t=C_t$,  with $C_t=2^{T-t}$ for all $t\in\mathcal{T}$. Thus, $C_1=2^{T-1},\ 
C_2=2^{T-2}, 
\ldots, 
C_T=1$.

Suppose that demand occurs only in the final period and equals total
available capacity, i.e., $d_t=0$ for all $t\in\mathcal{T}-1$, and $d_T
=
\sum_{t=1}^{T}C_t
=
2^T-1$.   Then $D_0=D_1=\cdots=D_{T-1}=0$.

Consider the positive grid at period $T-1$. Since $\mathcal{B}_t=\{0,C_t\}$,  we have
\begin{equation*}
G_{T-1}^{+}
\supseteq
\left\{
\sum_{u\in S}C_u
:
S\subseteq\{1,\ldots,T-1\}
\right\}. 
\end{equation*}

Because $C_1,\ldots,C_{T-1}$ are distinct powers of two, every subset produces a different sum.
Consequently, $|G_{T-1}|
\geq
2^{T-1}$.

Hence, the additive grid may have exponential cardinality.
\end{example}

Example~\ref{ex:exponential-grid} reveals an important distinction. The MGT is an \emph{optimality theorem}: it identifies a discrete state space containing an optimal solution. It is not, by itself, a polynomiality theorem.


The computational question is therefore:

\begin{quote}
\centering
\emph{Under what structural conditions can the additive boundary sums, and hence the cumulative production grids, be represented by only polynomially many states?}
\end{quote}

This question will lead, after the  construction of the dynamic programming algorithm of the next section, to the notions of polynomial grids and additive dimension in Section \ref{sec:polynomial-grid}.

\section{Grid Dynamic Programming}
\label{sec:grid-dp}

The MGT reduces the search for an optimal solution to a
finite collection of cumulative production levels. Once the grids
$G_0,\ldots,G_T$ are known, the resulting dynamic programming
formulation is immediate: a state records the current period and the
cumulative production reached at the end of that period.

This section presents the corresponding dynamic program, its
shortest path interpretation, and an accelerated implementation whose
running time is linear in the total number of grid states, up to the
cost of sorting the grids.

We call this algorithm: the \emph{Grid Dynamic Programming} (GDP). 

\subsection{State space}
\label{subsec:grid-dp-states}

For every period $t\in\mathcal{T}\cup\{0\}$ and every
$s\in G_t$, define $F_t(s)$ as the minimum cost incurred over periods $1,\ldots,t$ among all
feasible partial production plans satisfying $P_t=s$ and $P_u\in G_u$ for all $u=0,\ldots,t$.  If no such partial plan exists, we set $F_t(s)=+\infty$.

The initial conditions are $F_0(0)=0$ and $F_0(s)=+\infty$ for all $s\in G_0\setminus\{0\}$.

A transition from state $r\in G_{t-1}$ to state $s\in G_t$
corresponds to the production quantity $x_t=s-r$.  Such a transition is feasible if either $s-r=0$, or $L_t\leq s-r\leq U_t$ where $s-r>0$.  The second condition simply expresses the fact that positive
production must lie between the lower and upper production boundaries.

For $x\geq 0$, define the one-period setup and production cost
\begin{equation*}
\phi_t(x)
:=
\begin{cases}
0,
& x=0,
\\[1mm]
f_t+p_t x,
& x>0 \text{ and } L_t\leq x\leq U_t,
\\[1mm]
+\infty,
& \text{otherwise}.
\end{cases} 
\end{equation*}

Since all setup costs are nonnegative, no setup needs to be activated
when $x_t=0$. Thus, $\phi_t$ correctly accounts for both setup and
variable production costs.

\subsection{Dynamic programming recurrence}
\label{subsec:grid-dp-recurrence}

If the cumulative production level at the end of period $t$ is $s$, then the end-of-period inventory is $I_t=s-D_t$,   and the associated holding cost is $h_t(s-D_t)$. By construction, $G_t\subseteq[D_t,D_T]$, so every state $s\in G_t$ satisfies $I_t=s-D_t\geq 0$.

For every $t\in\mathcal{T}$ and every $s\in G_t$, the dynamic
programming recurrence is therefore
\begin{equation}
F_t(s)
=
h_t(s-D_t)
+
\min_{r\in G_{t-1}}
\left\{
F_{t-1}(r)
+
\phi_t(s-r)
\right\}.
\label{eq:grid-dp-recurrence}
\end{equation}

The terminal cumulative production level is fixed by $P_T=D_T$.  Hence, the value returned by the dynamic program is $F_T(D_T)$.

Whenever a minimizing predecessor is stored for every finite state, an
optimal production plan can be recovered by tracing the predecessors
backward from $(T,D_T)$ and setting $x_t=P_t-P_{t-1}$.

The correctness of the dynamic program follows directly from the MGT.

\begin{theorem}
\label{thm:grid-dp}
The dynamic programming recurrence
\eqref{eq:grid-dp-recurrence} returns the optimal value of
problem~$\mathcal P$. 
\end{theorem}

\begin{proof}
We first consider the problem restricted to cumulative production
trajectories satisfying $P_t\in G_t$ for all $t\in\mathcal{T}\cup\{0\}$.

For $t=0$, the only feasible cumulative production level is $P_0=0$,
so the initialization is correct. 

Assume that the dynamic programming values are correct at the end of
period $t-1$, and consider a state $s\in G_t$. Every feasible partial
solution ending at $s$ has a predecessor $r=P_{t-1}\in G_{t-1}$,  and produces $x_t=s-r$.  The transition is feasible precisely when $x_t=0$ or $L_t\leq x_t\leq U_t$ for $x_t>0$. Its additional cost is $\phi_t(x_t)+h_t(s-D_t)$.  Taking the minimum over all feasible predecessors therefore gives
exactly recurrence~\eqref{eq:grid-dp-recurrence}.

Therefore, by induction, $F_T(D_T)$ is the optimal value among all feasible
solutions whose cumulative production levels belong to the grids. Moreover,  by the MGT,  restricting the dynamic program to the additive grids causes no loss of optimality. Thus, the solution of the dynamic program is optimal.   
\end{proof}

\subsection{Shortest path interpretation}
\label{subsec:shortest-path}

The dynamic program admits an equivalent shortest path interpretation.

Construct a layered directed acyclic graph. For every
$t\in\mathcal{T}\cup\{0\}$ and every $s\in G_t$, create a node $(t,s)$, such that the initial node is $(0,0)$,  and the terminal node is $(T,D_T)$.

For every $t\in\mathcal{T}$, draw an arc from $(t-1,r)$  to $(t,s)$ if and only if
\begin{equation*}
s-r=0
\quad\text{or}\quad
\left(
s-r>0
\text{ and }
L_t\leq s-r\leq U_t
\right). 
\end{equation*}

The cost of this arc is
\begin{equation*}
c_t(r,s)
=
\phi_t(s-r)+h_t(s-D_t). 
\end{equation*}

Every path from $(0,0)$ to $(T,D_T)$ defines a production plan through $x_t=P_t-P_{t-1}$.  Conversely, every feasible production plan whose cumulative production
trajectory belongs to the grids defines such a path.

Moreover, the total cost of the path is exactly the objective value of
the corresponding production plan. Hence, the GDP algorithm is equivalent to a shortest path problem in a
layered acyclic graph.

\subsection{Complexity}
\label{subsec:grid-dp-complexity}

Let $N_t:=|G_t|$, $t\in\mathcal{T}\cup\{0\}$,  and define $N:=\max_{0\leq t\leq T}N_t$.

A direct implementation of
recurrence~\eqref{eq:grid-dp-recurrence} examines every pair $(r,s)\in G_{t-1}\times G_t$  and therefore has running time
\begin{equation*}
O
\left(
\sum_{t=1}^{T}N_{t-1}N_t
\right)
\subseteq
O(TN^2). 
\end{equation*}

However, the special structure of the transition costs allows a substantially
faster implementation.

For $s\in\mathbb{R}$, define the extension
\begin{equation*}
\overline{F}_{t-1}(s)
:=
\begin{cases}
F_{t-1}(s),
& s\in G_{t-1},
\\
+\infty,
& s\notin G_{t-1}.
\end{cases} 
\end{equation*}
Separating zero production from positive production,
recurrence~\eqref{eq:grid-dp-recurrence} can be written as
\begin{equation}
F_t(s)
=
h_t(s-D_t)
+
\min
\left\{
\overline{F}_{t-1}(s),
\;
f_t+p_t s+
\min_{r\in W_t(s)}
\left\{
F_{t-1}(r)-p_t r
\right\}
\right\},
\label{eq:accelerated-recurrence}
\end{equation}
where
\begin{equation*}
W_t(s)
:=
\left\{
r\in G_{t-1}:
s-U_t\leq r\leq s-L_t,
\;
r<s
\right\}. 
\end{equation*}

Indeed, for positive production, $x_t=s-r$,  and therefore
\[
f_t+p_t(s-r)
=
f_t+p_t s-p_t r.
\]

For a fixed period $t$, suppose that $G_{t-1}$ and $G_t$ are sorted in increasing order. As $s$ moves through the ordered elements of $G_t$, both endpoints of the predecessor interval $[s-U_t,s-L_t]$ move monotonically to the right. Consequently, the quantities $F_{t-1}(r)-p_t r$ over the moving predecessor window $W_t(s)$ can be maintained using a monotone double-ended queue (deque). Indeed, let
$$
Q_t(r)=F_{t-1}(r)-p_t r.
$$
When a new predecessor $r_2$ enters the window, any earlier predecessor $r_1<r_2$ with $Q_t(r_2)\leq Q_t(r_1)$ can be discarded permanently. $r_2$ has no larger value and, because it lies to the right of $r_1$, it leaves every subsequent moving window no earlier than $r_1$. The deque therefore stores only non-dominated predecessors, in increasing order of both their grid position and their $Q_t$-value. Its first element gives the minimum over the current window. Each predecessor grid point enters and leaves the
deque at most once.

This gives the following complexity result.

\begin{proposition} 
\label{prop:grid-dp-complexity}
Assume that the period-dependent grids $G_0,\ldots,G_T$ are given and sorted. Then the GDP algorithm can be implemented in
\begin{equation}
O
\left(
\sum_{t=1}^{T}
(N_{t-1}+N_t)
\right)
\subseteq
O(TN)
\label{eq:accelerated-complexity}
\end{equation}
time and $O(N)$ working memory for computing the optimal value. 
\end{proposition}

\begin{proof}
Fix a period $t$. The states of $G_t$ and the predecessor states of $G_{t-1}$ are scanned in increasing order.

For the positive production term in \eqref{eq:accelerated-recurrence}, a monotone deque maintains the minimum of $F_{t-1}(r)-p_t r$ over the moving window $W_t(s)$. Each point of $G_{t-1}$ is inserted into the deque at most once and removed at most once. Hence, evaluating all positive production transitions in period $t$ requires $O(N_{t-1}+N_t)$ time. The zero-production term $\overline{F}_{t-1}(s)$ can be evaluated during the same ordered scan. Summing over all periods gives~\eqref{eq:accelerated-complexity}. 

Only two consecutive dynamic programming layers and the monotone deque need to be stored when only the optimal value is required. Each contains at most $O(N)$ values, yielding the stated working-memory bound. If a predecessor is stored for every finite state in order to recover an optimal solution without recomputation, the corresponding storage is $O(\sum_{t=0}^{T}N_t)$.
\end{proof}

\sloppy If the period-dependent grids are not initially sorted, the total
sorting cost is $O
\left(
\sum_{t=0}^{T}N_t\log N_t
\right)$. Thus, including sorting, the dynamic programming phase requires 

\begin{equation*}
    O\left(\sum_{t=0}^{T}N_t\log N_t + \sum_{t=1}^{T}(N_{t-1}+N_t)\right),
\end{equation*}

which is

\begin{equation*}
O
\left(
\sum_{t=0}^{T}N_t\log N_t
+
TN
\right). 
\end{equation*} 
Alternatively, one may use the global additive grid $G=\bigcup_{t=0}^{T}G_t$ as a common state space in every period. 

Let $N_G:=|G|$. After sorting $G$ once, the same sliding window argument yields a
running time of
\begin{equation*}
O
\left(
N_G\log N_G+TN_G
\right) 
\end{equation*}
and $O(N_G)$ working memory for computing the optimal value.

The period-dependent formulation may use fewer states, whereas the
common grid formulation is simpler and will be particularly convenient
for the polynomiality analysis developed in   Section~\ref{sec:polynomial-grid}.

The complexity of the GDP algorithm is therefore
controlled by a single structural quantity: the number of additive
grid points that must be represented. The MGT guarantees
that these points contain an optimal cumulative production trajectory,
but, as shown in Example~\ref{ex:exponential-grid}, their number may
still be exponential.

This leads to the second fundamental question of Grid Theory:
\begin{quote}
\centering
\emph{Under what conditions does the additive structure of the
production boundaries generate a polynomial size grid?}
\end{quote}

The next section addresses this question through the notions of
polynomial grids and additive dimension.

\section{Polynomial Grid Theory}
\label{sec:polynomial-grid}

This section develops the second main component of Grid Theory. We
first formalize the polynomial grid condition. We then introduce the
notion of additive dimension, which measures the number of aggregate
coordinates required to represent the production boundaries. Bounded
additive dimension will provide a general sufficient condition for
polynomial solvability.

Throughout this section, all numerical data are assumed to be rational
and encoded in binary. Arithmetic operations and comparisons are
performed exactly.

\subsection{The polynomial grid condition}
\label{subsec:polynomial-grid-condition}

Let $\mathfrak{I}$ be a family of instances of
problem~$\mathcal P$, and let $\langle I\rangle$ denote the
binary encoding length of an instance $I\in\mathfrak{I}$.

For a given instance, let $G_0,\ldots,G_T$ be the period-dependent additive grids defined in
Section~\ref{sec:grid-theory}, and let $G=\bigcup_{t=0}^{T}G_t$ be the corresponding global grid.

It is not necessary to construct $G$ exactly. A larger set may be used
as long as it contains all grid points required by the MGT.

\begin{definition} 
\label{def:grid-envelope}
A finite set $H\subseteq[0,D_T]$ is called a \emph{grid envelope} for an instance $I$ if $G\subseteq H$.  
\end{definition}

A grid envelope may contain points that do not belong to the exact additive grid and may even contain cumulative production levels that cannot occur in any feasible solution. Such additional states do not affect correctness, since only states $s\geq D_t$ are considered in period $t$ and production feasibility is still enforced by the dynamic programming transitions.

\begin{proposition} 
\label{prop:grid-enlargement}
Let $H$ be a grid envelope. If the GDP algorithm is applied using $H\cap[D_t,D_T]$ in period $t$, then it returns the optimal value of problem~$\mathcal P$.
\end{proposition}

\begin{proof}
Since $G\subseteq H$, the MGT guarantees the existence of an optimal solution whose cumulative production trajectory is entirely contained in $H$. Since $G_t\subseteq[D_t,D_T]$, this trajectory is also contained in $H\cap[D_t,D_T]$ in every period $t$.

Conversely, every finite-cost path considered by the dynamic program satisfies the original production boundary and inventory constraints. The restriction $s\geq D_t$ ensures nonnegative inventory, while the transition conditions enforce production feasibility.
Thus, enlarging the grid may introduce additional states, but it cannot introduce a solution with an objective value smaller than the true optimum.

Hence, the dynamic program over $H$ contains at least one optimal trajectory and returns the optimal value.
\end{proof}

This observation allows polynomiality to be formulated in terms of any
efficiently constructible envelope rather than the exact grid.

\begin{definition} 
\label{def:polynomial grid-condition}
A family $\mathfrak{I}$ of dynamic lot-sizing instances is said to satisfy the \emph{polynomial grid condition} if there exists a polynomial $q$ such that, for every instance $I\in\mathfrak{I}$, one can construct in time polynomial in $\langle I\rangle$, a grid envelope $H(I)$ satisfying $|H(I)|\leq q(\langle I\rangle)$, with all grid values represented exactly.
\end{definition}

The polynomial grid condition immediately yields a polynomial algorithm.

\begin{proposition} 
\label{prop:polynomial grid-principle}
Every family of instances satisfying the polynomial grid condition is
solvable in polynomial time by the GDP algorithm.
\end{proposition}

\begin{proof}
Let $H(I)$ be the grid envelope provided by
Definition~\ref{def:polynomial grid-condition}, and set $N_H:=|H(I)|$.  By assumption, $N_H$ is polynomially bounded in $\langle I\rangle$,
and the grid can be generated in polynomial time.

A polynomial number of rational grid points can also be sorted in
polynomial time. By Proposition~\ref{prop:grid-dp-complexity}, the
dynamic programming phase requires $O(N_H\log N_H+TN_H)$ time when the common grid implementation is used. Since $T$ is polynomially bounded by $\langle I\rangle$ and $N_H$ is polynomially bounded, the complete algorithm runs in polynomial time.
\end{proof}

Definition~\ref{def:polynomial grid-condition} is deliberately stated
in terms of a grid envelope. Exact enumeration of all boundary sums is
not required. The central issue is whether the additive structure can
be represented by a polynomial number of states while preserving at
least one optimal trajectory.

We next introduce a structural condition that guarantees precisely
this property.

\subsection{Additive dimension}
\label{subsec:additivedimension}

Let the collection $\mathcal{B}
:=
(\mathcal{B}_1,\ldots,\mathcal{B}_T)$  be called the \emph{boundary resource profile} of the instance.

The exact grids introduced in Section~\ref{sec:grid-theory} are
generated by sums of elements of these sets. Their size is therefore
determined by the additive complexity of the boundary resource profile.

The following definition makes this idea precise.

\begin{definition} 
\label{def:additive-representation}
A family of boundary resource profiles is said to admit a
$p$-dimensional additive representation if there exist fixed
nonnegative integers $c_1,\ldots,c_p$, independent of the planning horizon, such that every profile in the
family admits rational basis values $\theta_1,\ldots,\theta_p$ and integer-valued coordinate maps $g_i(t,b)$,  with $i=1,\ldots,p$, $t\in\mathcal{T}$, and $b\in\mathcal{B}_t$,   satisfying
\begin{equation}
b
=
\sum_{i=1}^{p}
\theta_i g_i(t,b),
\qquad
t\in\mathcal{T},
\quad
b\in\mathcal{B}_t,
\label{eq:additive-representation}
\end{equation}
with $|g_i(t,b)|
\leq
T^{c_i}$ for all $i=1,\ldots,p$, and we additionally set $g_i(t,0)=0$ for all $i=1,\ldots,p$. 
 
We further require that the basis values $\theta_i$ and the coordinates $g_i(t,b)$ can be computed in time polynomial in the binary encoding length of the instance, and that their binary encoding lengths are polynomially bounded.
 
\end{definition}

The basis values $\theta_i$ may depend on the instance. What is fixed over the family is the number $p$ of additive coordinates and the polynomial growth rates of these coordinates.

This leads to the central structural parameter.

\begin{definition} 
\label{def:additive dimension}
The \emph{additive dimension} of a family of boundary resource
profiles, denoted by $\operatorname{adim}(\mathcal{B})$,  is the smallest integer $p$ for which the family admits a
$p$-dimensional additive representation in the sense of
Definition~\ref{def:additive-representation}.

A family is said to have \emph{bounded additive dimension} if this
number is bounded independently of the planning horizon.
\end{definition}

\begin{remark} 
The term \enquote{additive dimension} used here should not be confused
with the notion bearing the same name in  \textit{additive combinatorics} and
\textit{additive number theory}.   The notion introduced in
Definition~\ref{def:additive dimension} is specific to the present
framework.  
\end{remark}

Additive dimension measures the number of aggregate coordinates needed
to describe the resource quantities that generate the grid. It should
not be confused with the number of periods or with the number of
possible subsets of periods.

To see its significance, consider an arbitrary boundary sum $a
=
\sum_{u\in J}b_u$ where $b_u\in\mathcal{B}_u$ and   $J\subseteq\mathcal{T}$. Using~\eqref{eq:additive-representation}, we obtain
\begin{equation}
a
=
\sum_{u\in J}
\sum_{i=1}^{p}
\theta_i g_i(u,b_u)
=
\sum_{i=1}^{p}
\theta_i\beta_i,
\label{eq:aggregate-representation}
\end{equation}
where $\beta_i
:=
\sum_{u\in J}g_i(u,b_u)$ for all $i=1,\ldots,p$.

Since $|J|\leq T$ and
$|g_i(u,b_u)|\leq T^{c_i}$, then $|\beta_i|
\leq
T^{c_i+1}$ for all $i=1,\ldots,p$.

Thus, exponentially many selections of boundary quantities are
compressed into the fixed-dimensional integer vector $(\beta_1,\ldots,\beta_p)$,   whose coordinates range only over polynomially bounded intervals. This is the fundamental compression mechanism underlying polynomial
Grid Theory.

To formalize it, define $U^{\Sigma}
:=
\sum_{t=1}^{T}U_t$. Then, every nonnegative boundary sum belongs to the following set:
\begin{equation*}
\widehat{\mathcal{A}}
:=
\left\{
\sum_{i=1}^{p}\theta_i\beta_i
:
\beta_i\in\mathbb{Z},
\;
|\beta_i|\leq T^{c_i+1},
\;
i=1,\ldots,p
\right\}
\cap[0,U^{\Sigma}]. 
\end{equation*}

The set $\widehat{\mathcal{A}}$ need not consist only of realizable
boundary sums. It is an additive envelope containing every sum that
can occur in the positive or negative grids.

\begin{lemma} 
\label{lem:additive-compression}
Suppose that the boundary resource profile admits a $p$-dimensional
additive representation. Then every boundary sum set
$\mathcal{A}(J)$ satisfies $\mathcal{A}(J)
\subseteq
\widehat{\mathcal{A}}$ for all $J\subseteq\mathcal{T}$.  Moreover,
\begin{equation*}
|\widehat{\mathcal{A}}|
\leq
\prod_{i=1}^{p}
\left(
2T^{c_i+1}+1
\right). 
\end{equation*}
\end{lemma}

\begin{proof}
Let $a\in\mathcal{A}(J)$. Then $a=\sum_{u\in J}b_u$  for some $b_u\in\mathcal{B}_u$. By~\eqref{eq:aggregate-representation}, $a
=
\sum_{i=1}^{p}\theta_i\beta_i$,  where the aggregate coordinates satisfy $|\beta_i|\leq T^{c_i+1}$. 
 
Furthermore, because all production boundaries are nonnegative, we have  $0\leq a\leq U^\Sigma$. Therefore, $a\in\widehat{\mathcal{A}}$.

For each coordinate $\beta_i$, there are at most $2T^{c_i+1}+1$  integer values satisfying its bound. Hence, the total number of
coordinate vectors is at most $\prod_{i=1}^{p}
\left(
2T^{c_i+1}+1
\right)$. 
 
Different vectors may generate the same value, so this is also an
upper bound on
$|\widehat{\mathcal{A}}|$.
\end{proof}

Using this additive envelope, define the common cumulative production
grid
\begin{equation*}
\widehat{G}
:=
\left\{
D_j+\sigma a
:
j\in\{0,\ldots,T\},
\;
\sigma\in\{-1,1\},
\;
a\in\widehat{\mathcal{A}}
\right\}
\cap[0,D_T]. 
\end{equation*}
The exact additive grid is contained in this set.

\begin{lemma} 
\label{lem:grid-compression}
If the boundary resource profile admits a $p$-dimensional additive
representation, then $G\subseteq\widehat{G}$.   Moreover,
\begin{equation}
|\widehat{G}|
\leq
2(T+1)
\prod_{i=1}^{p}
\left(
2T^{c_i+1}+1
\right).
\label{eq:aggregate-grid-cardinality}
\end{equation}
\end{lemma}

\begin{proof}
Let $s\in G$. Then $s\in G_t$ for some
$t\in\mathcal{T}\cup\{0\}$. If $s\in G_t^+$, then by definition $s
=
D_j+a$ for some $0\leq j\leq t$ and some $a\in
\mathcal{A}(\{j+1,\ldots,t\})$.  By Lemma~\ref{lem:additive-compression}, $a\in\widehat{\mathcal{A}}$,  and therefore $s\in\widehat{G}$.

Similarly, if $s\in G_t^-$, then $s
=
D_j-a$ for some $t\leq j\leq T$ and $a\in
\mathcal{A}(\{t+1,\ldots,j\})$.  Again, $a\in\widehat{\mathcal{A}}$,  and hence $s\in\widehat{G}$.  Thus, $G\subseteq\widehat{G}$.

For every cumulative demand anchor $D_j$, there are at most
$2|\widehat{\mathcal{A}}|$ values corresponding to the two signs.
Hence, $|\widehat{G}|
\leq
2(T+1)|\widehat{\mathcal{A}}|$.  

The result follows from
Lemma~\ref{lem:additive-compression}.
\end{proof}

The preceding lemmas reveal the role of additive dimension. The exact
number of subsets of production boundaries may be exponential, but
the number of aggregate descriptions of those subsets is polynomial
whenever the additive dimension is bounded.

\begin{remark} 
\label{rem:LiLi}
Li and Li \cite{LiLi2016} provide a general polynomial solvability
framework based on a fixed set of time-independent
\emph{special replenishment quantities} and the property that at most
one period between consecutive regeneration periods uses a quantity
outside this set. A fixed collection $\{q_1,\ldots,q_R\}$ of boundary values fits naturally within the present framework by taking the $q_r$ as basis values and using count coordinates. Hence, the setting with a finite set of special quantities constitutes an important special case of the additive representation considered here.

The converse need not hold. Bounded additive dimension does not require the number of distinct boundary values to remain fixed. For example, a fixed-degree polynomial capacity profile $C_t=\sum_{\ell=0}^{d}\gamma_\ell t^\ell$ may contain $T$ distinct capacity values, while every capacity sum is described by the fixed collection of power sum coordinates $\sum_{t\in S}t^\ell$ for  $\ell=0,\ldots,d$ (cf., Section~\ref{subsec:polynomial-capacities}). Thus, additive dimension captures structured resource families that need not admit a fixed list of special replenishment quantities.
\end{remark}

\subsection{The General Polynomiality Theorem}
\label{subsec:general-polynomiality}

We can now state the second main theorem of Grid Theory, which we call the \textit{General Polynomiality Theorem} (GPT). 

\begin{theorem} 
\label{thm:general-polynomiality}
Consider a family $\mathfrak{I}$ of feasible dynamic lot-sizing
instances of the form~$\mathcal P$. If the corresponding
boundary resource profiles have bounded additive dimension, then
$\mathfrak{I}$ satisfies the polynomial grid condition.  Consequently, every such instance can be solved exactly in polynomial
time by the GDP algorithm.
\end{theorem}

\begin{proof}
Let $p$ denote the bounded additive dimension, and let
$c_1,\ldots,c_p$ be the fixed exponents of an additive
representation. By Lemma~\ref{lem:grid-compression}, the exact additive grid is
contained in the grid envelope $\widehat{G}$, whose cardinality
satisfies~\eqref{eq:aggregate-grid-cardinality}. Since $p$ and $c_1,\ldots,c_p$ are fixed independently of $T$, this
quantity is polynomial in $T$. More precisely, $|\widehat{G}|
=
O
\left(
T^{
1+\sum_{i=1}^{p}(c_i+1)
}
\right)$.  

By Definition \ref{def:additive-representation}, the basis values and coordinate maps are polynomial-time computable and have polynomially bounded encoding lengths. The additive envelope $\widehat{\mathcal{A}}$ can therefore be generated by enumerating all
integer vectors $(\beta_1,\ldots,\beta_p)$ satisfying $|\beta_i|\leq T^{c_i+1}$ for all $i=1,\ldots,p$, and evaluating $\sum_{i=1}^{p}\theta_i\beta_i$ exactly.

Because the number of such vectors is polynomial, the additive
envelope can be generated in polynomial time. Combining each retained
value with the $T+1$ cumulative demand anchors and the two possible
signs generates $\widehat{G}$ in polynomial time.

Therefore, $\widehat{G}$ satisfies the polynomial grid condition. Moreover, by Proposition~\ref{prop:polynomial grid-principle}, the GDP algorithm solves the problem in polynomial time.
\end{proof}

The GPT provides the central answer to
the complexity question studied in this paper. More explicitly, let $\kappa
:=
\sum_{i=1}^{p}(c_i+1)$.  Then $|\widehat{G}|=O(T^{1+\kappa})$,  and the common grid implementation of
Section~\ref{sec:grid-dp} runs in
\begin{equation*}
O
\left(
T^{1+\kappa}\log T
+
T^{2+\kappa}
\right)
=
O(T^{2+\kappa}) 
\end{equation*}
arithmetic operations and comparisons. Under the binary encoding assumptions above, the corresponding exact rational operations and grid generation also require polynomial time.

The GPT  should be interpreted as a sufficient structural theorem, not as a characterization of the intrinsic complexity of every dynamic lot-sizing problem. A family whose boundary resource profile does not have bounded additive dimension may still admit a polynomial time algorithm through some other structural property. Similarly, an exponentially large exact grid does not imply computational intractability; it only shows that explicit enumeration of that grid does not provide a polynomial algorithm.

The importance of additive dimension is that it provides a common
explanation for a broad collection of polynomially solvable families.
Rather than designing an unrelated algorithm for each particular
resource pattern, one can establish a low-dimensional additive
representation and invoke the GPT.

\subsection{Extension to separable concave costs}
\label{subsec:concave costs}

The development so far has considered linear production and inventory costs. The grid structure is not restricted to linear costs. It follows from the geometry of the feasible region whenever the objective admits an optimal extreme point solution.


To make this distinction explicit, consider the following extension of
problem~$\mathcal P$. For each period $t$, let
$c_t(x_t)$ denote the production cost and let $H_t(I_t)$ denote the
inventory cost. We denote this problem by $\mathcal P^{\mathrm{conc}}$. This problem is given by 
\begin{align}
\min \quad &
\sum_{t=1}^{T}
\left(
f_t y_t+c_t(x_t)+H_t(I_t)
\right)\label{problem2:Objective}
\\
\text{\rm s.t., }&\text{\eqref{problem:balance}--\eqref{problem:zero-inventory}}\nonumber 
\end{align} 

We assume that $c_t$ and $H_t$ are concave on their respective domains and that $c_t(0)=H_t(0)=0$ without loss of generality.
 
For the algorithmic statements below, we additionally assume that their values at rational arguments can be evaluated exactly in polynomial time and have polynomially bounded binary encoding length. This condition holds, for example, for piecewise linear concave functions described by a polynomial number of rational breakpoints.

The following elementary observation is the key to the extension.

\begin{lemma} 
\label{lem:concave-extreme point}
Fix a setup vector $y\in\{0,1\}^{T}$. If $c_t$ and $H_t$ are concave for every $t$, then the fixed-setup problem associated with $\mathcal P^{\mathrm{conc}}$ admits an optimal solution that is an extreme point of $\mathcal{X}(y)$.
\end{lemma}

\begin{proof}
For a fixed setup vector $y$, the feasible production vectors belong to the polytope $\mathcal{X}(y)$ defined in \eqref{eq:fixed-setup-polytope}. Using $I_t=\sum_{u=1}^{t}x_u-D_t$,  the objective, apart from the constant setup term $\sum_{t=1}^{T}f_ty_t$, can be written as
\begin{equation*}
\Psi_y(x)
=
\sum_{t=1}^{T}
\left[
c_t(x_t)
+
H_t
\left(
\sum_{u=1}^{t}x_u-D_t
\right)
\right]. 
\end{equation*}
Since each $c_t$ and $H_t$ is concave and the inventory levels are affine functions of $x$, the function $\Psi_y$ is concave on
$\mathcal{X}(y)$.

Let $x^\star$ be an optimal solution. Since $\mathcal{X}(y)$ is a polytope, $x^\star$ can be expressed as a convex combination of its
extreme points: $x^\star = \sum_{q=1}^{m}\lambda_q v^q$ where $\lambda_q\geq0$ and  $\sum_{q=1}^{m}\lambda_q=1$.
 
By concavity,
$$
\Psi_y(x^\star)
\geq
\sum_{q=1}^{m}
\lambda_q\Psi_y(v^q).
$$
Therefore, at least one extreme point $v^q$ satisfies
$$
\Psi_y(v^q)
\leq
\Psi_y(x^\star).
$$
Since $x^\star$ is optimal, equality must hold. Hence, an optimal extreme point exists.
\end{proof}

Lemma~\ref{lem:concave-extreme point} shows that the extreme point boundary argument extends from linear to separable concave production and inventory costs.


\begin{theorem} 
\label{thm:concave-grid}
Suppose that $c_t$ and $H_t$ are concave for every
$t\in\mathcal{T}$. Then problem~$\mathcal P^{\mathrm{conc}}$ admits an
optimal solution such that, in every regeneration interval, at most
one production quantity is strictly between its production
boundaries. Consequently, there exists an optimal solution whose
cumulative production trajectory satisfies $P_t\in G_t$ for all $t\in\mathcal{T}\cup\{0\}$,   where $G_t$ is the same additive grid defined in
Section~\ref{sec:grid-theory}.
\end{theorem}

\begin{proof}

Let $y^\star$ be the setup vector of an optimal solution of problem~$\mathcal P^{\mathrm{conc}}$. By Lemma~\ref{lem:concave-extreme point}, there exists an optimal production vector $x^\star$ that is an extreme point of $\mathcal{X}(y^\star)$.

Theorem \ref{thm:extreme point} therefore implies that, in every regeneration interval, all but at most one production quantity belong to $\mathcal{B}_t=\{0,L_t,U_t\}$.

The proof of the MGT is purely structural and applies to any solution satisfying this boundary property. Hence, the corresponding cumulative production trajectory satisfies $P_t\in G_t$ for all  $t\in\mathcal{T}\cup\{0\}$. 
\end{proof}

The GDP algorithm can also be extended directly.
Define
\begin{equation*}
\phi_t^{\mathrm{conc}}(x)
:=
\begin{cases}
0,
& x=0,
\\[1mm]
f_t+c_t(x),
& x>0 \text{ and } L_t\leq x\leq U_t,
\\[1mm]
+\infty,
& \text{otherwise}.
\end{cases} 
\end{equation*}
For every $s\in G_t$, the recurrence becomes
\begin{equation}
F_t(s)
=
H_t(s-D_t)
+
\min_{r\in G_{t-1}}
\left\{
F_{t-1}(r)
+
\phi_t^{\mathrm{conc}}(s-r)
\right\}.
\label{eq:concave-grid-dp}
\end{equation}

Unlike the linear cost recurrence \eqref{eq:accelerated-recurrence}, the term $c_t(s-r)$ cannot in general be decomposed into a part depending only on $s$ and a part depending only on $r$. Consequently, the monotone deque acceleration of Section~\ref{subsec:grid-dp-complexity} does not apply in general.

Let $N:=\max_{0\leq t\leq T}|G_t|$.  A direct evaluation of~\eqref{eq:concave-grid-dp} requires a running time 
\begin{equation*}
O
\left(
\sum_{t=1}^{T}
|G_{t-1}||G_t|
\right)
\subseteq
O(TN^2). 
\end{equation*} 
Combining this observation with the polynomial grid theory yields the
following consequence.

\begin{corollary} 
\label{cor:concave-polynomiality}
Consider a family of instances of
problem~$\mathcal P^{\mathrm{conc}}$ whose boundary resource profiles have
bounded additive dimension. If the concave cost functions can be
evaluated exactly in polynomial time, then the family is solvable
exactly in polynomial time by GDP.
\end{corollary}

\begin{proof}
By Theorem~\ref{thm:concave-grid}, an optimal solution exists on the
same additive grid as in the linear cost model. Moreover, by Theorem~\ref{thm:general-polynomiality}, bounded additive dimension
provides a polynomially constructible grid envelope
$\widehat{G}$ of polynomial cardinality. 

Applying the recurrence~\eqref{eq:concave-grid-dp} to the feasible states $\widehat{G}\cap[D_t,D_T]$ in period $t$ examines at most $|\widehat{G}|^2$ predecessor--successor pairs per period.
Therefore, the total number of transitions is polynomial in the input size. Since every transition cost can also be evaluated in polynomial time, the complete algorithm is polynomial.
\end{proof}

More explicitly,  
the direct concave cost dynamic program   has the generic
bound
$
O
\left(
T|\widehat{G}|^2
\right)
=
O
\left(
T^{3+2\kappa}
\right), 
$
up to the polynomial cost of grid generation, sorting, and exact evaluation of the cost functions. Thus, linear costs are needed for the accelerated $O(T|\widehat{G}|)$ implementation developed in Section~\ref{sec:grid-dp}, but they are not needed for either the MGT or the polynomiality mechanism itself.

\section{Polynomial Families}
\label{sec:polynomial-families}

The GPT provides a common mechanism for
establishing polynomial solvability. Rather than deriving a different
algorithm for every particular resource profile, it is sufficient to
identify a low-dimensional additive representation of the production
boundaries.

This section illustrates the resulting principle on several families
of dynamic lot-sizing problems. In each case, we identify the additive
coordinates, construct a polynomial grid envelope, and invoke the GPT. For some families, the generic bound of
the GPT can be sharpened by exploiting
additional structure. We do so whenever a simpler exact or enclosing
set of boundary sums is available.

\subsection{Uncapacitated baseline}
\label{subsec:uncapacitated}

Before considering structured capacity profiles, it is useful to
recover the classical uncapacitated lot-sizing problem. Set $L_t=0$ and introduce the redundant upper boundary $U_t=D_T$ for all $t\in\mathcal{T}$. Since total production over the horizon equals $D_T$, this upper
boundary does not remove any feasible solution.

In this case, the only relevant cumulative production levels generated
by the positive and negative grid constructions are cumulative demand
anchors. Hence, a common grid envelope is
\[
H^{\mathrm{unc}}
=
\{D_0,D_1,\ldots,D_T\},
\]
and therefore $|H^{\mathrm{unc}}|=O(T)$.

\begin{corollary} 
\label{cor:uncapacitated}
The uncapacitated single-item dynamic lot-sizing problem with linear costs can be solved by GDP in $O(T^2)$ time and $O(T)$ working memory for computing the optimal value.
\end{corollary}

\begin{proof}
The common cumulative production grid contains $O(T)$ states. The
accelerated recurrence of Section~\ref{sec:grid-dp} therefore requires
$O(T|H^{\mathrm{unc}}|)=O(T^2)$ time.
\end{proof}

This recovers, within the Grid Theory framework, the classical
quadratic order of the Wagner--Whitin dynamic program
\cite{WagnerWhitin1958}. The same collapse of cumulative production
states to cumulative demand levels is also observed in the
forward--backward construction of K{\i}rca \cite{Kirca1990}.

\subsection{Constant capacity}
\label{subsec:constant capacity}

Consider first the classical capacitated lot-sizing problem with
constant capacity: $L_t=0$ and  $U_t=C$ for all $t\in\mathcal{T}$.

The production boundary set is identical in every period, i.e., $\mathcal{B}_t=\{0,C\}$.  Every boundary sum is therefore of the form $kC$ with $k\in\{0,\ldots,T\}$.  Thus, the exact boundary sum set is contained in $\mathcal{A}^{C}
:=
\{kC:k=0,\ldots,T\}$,  which contains at most $T+1$ elements.

A common cumulative production grid envelope is consequently
\begin{equation*}
H^{C}
:=
\left\{
D_j+kC:
j\in\{0,\ldots,T\},
\;
-T\leq k\leq T
\right\}
\cap[0,D_T]. 
\end{equation*}
Hence,
\begin{equation*}
|H^{C}|
\leq
(T+1)(2T+1)
=
O(T^2). 
\end{equation*}
The same result can be viewed through additive dimension. Indeed, the boundary resource profile admits the one-dimensional representation $b=Cg(t,b)$,  where $g(t,0)=0$ and $g(t,C)=1$. Thus, $\operatorname{adim}(\mathcal{B}) \leq 1$. 

\begin{corollary} 
\label{cor:constant capacity}
The single-item dynamic lot-sizing problem with constant capacity can be solved by the GDP algorithm in $O(T^3)$ time and $O(T^2)$ working memory for computing the optimal value when the common grid $H^{C}$ is used.
\end{corollary}

\begin{proof}
The grid envelope contains $O(T^2)$ states. Once sorted, Proposition~\ref{prop:grid-dp-complexity} yields $O(T|H^{C}|) = O(T^3)$ time and $O(T^2)$ memory.
\end{proof}

The constant capacity case is classical. Florian and Klein \cite{FlorianKlein1971} obtained an $O(T^4)$ algorithm under concave production and inventory costs, while van Hoesel and Wagelmans \cite{vanHoeselWagelmans1996} subsequently obtained an $O(T^3)$ algorithm for concave production costs and linear holding costs. Corollary~\ref{cor:constant capacity} recovers the cubic order for the linear cost model through the one-dimensional additive structure of the capacity profile. Under the generic concave cost recurrence of Section~\ref{subsec:concave costs}, the same $O(T^2)$ grid leads to an $O(T^5)$ bound; this generic bound is not intended to improve the specialized constant capacity algorithms.

\subsection{Constant minimum order quantity}
\label{subsec:constant-moq}

We next consider an uncapacitated dynamic lot-sizing problem with a constant minimum order quantity $L>0$. Positive production must satisfy $x_t\geq L$.

Since total production over the complete horizon is $D_T$, no period needs to produce more than $D_T$. Therefore, without loss of generality, the model can be represented within problem~$\mathcal P$ by setting $L_t=L$  and $U_t=D_T$ for all $t\in\mathcal{T}$. The production boundary sets are then $\mathcal{B}_t=\{0,L,D_T\}$.

Consider a boundary sum that lies in the relevant interval $[0,D_T]$. If it contains no occurrence of $D_T$, it has the form $kL$ with $0\leq k\leq T$. If it contains one occurrence of $D_T$, then no positive additional boundary quantity can be added without exceeding $D_T$. Two occurrences of $D_T$ also exceed the relevant interval. It is therefore sufficient to use
\begin{equation*}
\mathcal{A}^{L}
:=
\{kL:k=0,\ldots,T\}
\cup
\{D_T\}. 
\end{equation*}
Hence, $|\mathcal{A}^{L}|=O(T)$.

A common grid envelope is
\begin{equation*}
H^{L}
:=
\left\{
D_j+\sigma a:
j\in\{0,\ldots,T\},
\;
\sigma\in\{-1,1\},
\;
a\in\mathcal{A}^{L}
\right\}
\cap[0,D_T]. 
\end{equation*}
Consequently, $|H^{L}|=O(T^2)$.

\begin{corollary} 
\label{cor:constant-moq}
The uncapacitated single-item dynamic lot-sizing problem with a constant minimum order quantity can be solved by the GDP algorithm in $O(T^3)$ time and $O(T^2)$ working memory for computing the optimal value.
\end{corollary}

\begin{proof}
The set $H^{L}$ is a grid envelope of cardinality $O(T^2)$. Proposition~\ref{prop:grid-dp-complexity} therefore gives an $O(T^3)$ implementation.
\end{proof}

This result illustrates the advantage of working with production boundaries rather than capacities alone. The same Grid Theory applies when the relevant resource is a lower production boundary.

Polynomial algorithms for uncapacitated lot-sizing models with minimum order quantities have also been developed previously. In particular, Okhrin and Richter \cite{OkhrinRichter2011b} obtain an $O(T^2)$ algorithm for a special linear MOQ model. The purpose of Corollary~\ref{cor:constant-moq} is therefore not to establish the first polynomiality result for a constant MOQ. Rather, it shows that a constant lower production boundary generates an $O(T^2)$ cumulative production grid under the general linear cost structure considered in this paper.

\subsection{Constant capacity and constant minimum order quantity}
\label{subsec:constant capacity-moq}

Consider now the simultaneous restrictions, i.e., $L_t=L$ and  $U_t=C$ with $0<L\leq C$ for all $t\in\mathcal{T}$.

The common production boundary set is $\mathcal{B}_t=\{0,L,C\}$. Every boundary sum can be written as $kL+\ell C$,  where $k$ denotes the number of periods contributing their lower boundary and $\ell$ the number of periods contributing their upper boundary. Necessarily, $k\geq0$, $\ell\geq0$, and $k+\ell\leq T$. Hence, it is sufficient to consider
\begin{equation*}
\mathcal{A}^{L,C}
:=
\left\{
kL+\ell C:
k,\ell\in\mathbb{Z}_+,\;
k+\ell\leq T
\right\}
\cap[0,D_T]. 
\end{equation*}
Its cardinality satisfies $|\mathcal{A}^{L,C}|
\leq
(T+1)^2
=
O(T^2)$, and the  corresponding common grid envelope is
\begin{equation*}
H^{L,C}
:=
\left\{
D_j+\sigma a:
j\in\{0,\ldots,T\},
\;
\sigma\in\{-1,1\},
\;
a\in\mathcal{A}^{L,C}
\right\}
\cap[0,D_T], 
\end{equation*}
and therefore $|H^{L,C}|=O(T^3)$.

Equivalently, the boundary profile admits a two-dimensional additive representation with basis values $L$ and $C$.

\begin{corollary} 
\label{cor:constant capacity-moq}
The single-item dynamic lot-sizing problem with constant capacity $C$ and constant minimum order quantity $L$ can be solved by the GDP algorithm in $O(T^4)$ time and $O(T^3)$ working memory for computing the optimal value.
\end{corollary}

\begin{proof}
The grid envelope $H^{L,C}$ contains $O(T^3)$ cumulative production states. The result follows directly from Proposition~\ref{prop:grid-dp-complexity}.
\end{proof}

The contrast between Corollaries~\ref{cor:constant capacity}, \ref{cor:constant-moq}, and \ref{cor:constant capacity-moq} is instructive. A single constant boundary resource produces an essentially one-dimensional family of boundary sums, whereas two independent constant boundary resources produce a two-dimensional family.

The constant capacity, constant MOQ family has been studied under several cost assumptions. Okhrin and Richter \cite{OkhrinRichter2011a} derive an $O(T^3)$ algorithm under a
restricted linear cost structure. Hellion, Mangione, and Penz \cite{HellionMangionePenz2012,HellionMangionePenz2013} consider concave production and storage costs and obtain a corrected $O(T^6)$ algorithm. Koca, Yaman, and Akturk \cite{KocaYamanAkturk2014} recover the same $O(T^6)$ order by viewing the lower and upper production limits as two time-invariant breakpoints of a piecewise concave production cost function. The general framework of Li and Li \cite{LiLi2016}, with the two special replenishment quantities $L$ and $C$, yields an $O(T^7)$ bound.

Corollary~\ref{cor:constant capacity-moq} should therefore be viewed as a unified consequence of Grid Theory rather than as the first polynomial algorithm for this family. Under the linear cost model studied here, the two additive coordinates $(k,\ell)$ yield an $O(T^3)$ grid and an $O(T^4)$ GDP algorithm. For separable concave costs, the generic recurrence of Section~\ref{subsec:concave costs} yields $O(T^7)$.

\begin{remark}
The terminal inventory conventions are not identical across all of these models. The formulation considered in this paper imposes $I_T=0$. Okhrin and Richter \cite{OkhrinRichter2011a} use the same zero-terminal-inventory convention, whereas minimum-order-quantity models allowing positive terminal inventory are also considered by Hellion, Mangione, and Penz \cite{HellionMangionePenz2012} and Park and Klabjan \cite{ParkKlabjan2015}.
\end{remark}

\subsection{A fixed number of capacity levels}
\label{subsec:fixed-capacity-levels}

Return to the classical capacitated model, i.e., $L_t=0$ and $U_t=C_t$ for all $t \in \mathcal{T}$,  and suppose that the capacities take values in a fixed set $\{q_1,\ldots,q_R\}$,  where $R$ is independent of the planning horizon.

For every $r=1,\ldots,R$, define $n_r
:=
\left|
\{t\in\mathcal{T}:C_t=q_r\}
\right|$. For a subset $S\subseteq\mathcal{T}$, let $\mu_r(S)
:=
\left|
\{t\in S:C_t=q_r\}
\right|$. Then
\begin{equation*}
\sum_{t\in S}C_t
=
\sum_{r=1}^{R}\mu_r(S)q_r, 
\end{equation*}
with $0\leq\mu_r(S)\leq n_r$.

Thus, every capacity subset sum is completely described by the $R$-dimensional count vector $(\mu_1,\ldots,\mu_R)$.

The exact subset sum set is contained in
\begin{equation*}
\mathcal{A}^{R}
:=
\left\{
\sum_{r=1}^{R}\mu_r q_r:
\mu_r\in\{0,\ldots,n_r\},
\;
r=1,\ldots,R
\right\}, 
\end{equation*}
and
\begin{equation*}
|\mathcal{A}^{R}|
\leq
\prod_{r=1}^{R}(n_r+1)
\leq
(T+1)^R. 
\end{equation*}

The boundary resource profile therefore satisfies $\operatorname{adim}(\mathcal{B}) \leq R$.  The common cumulative production grid has cardinality $|H^{R}| = O(T^{R+1})$. 

\begin{corollary} 
\label{cor:fixed-capacity-levels}
For every fixed $R$, the capacitated lot-sizing problem in which $C_t\in\{q_1,\ldots,q_R\}$ for all $t$ is solvable by the GDP algorithm in $O(T^{R+2})$ time.
\end{corollary}

The order in which the capacity levels occur is irrelevant to this
result. The profile may switch arbitrarily among the $R$ values.
Hence, the same corollary applies to a fixed number of shifts,
production modes, exogenous operating states, or recurring capacity
levels.

Polynomial solvability with a fixed number of distinct capacity
values also follows from more general frameworks. Li and Li
\cite{LiLi2016} treat the distinct capacities as special replenishment
quantities and obtain an $O(T^{2R+3})$ algorithm. In the broader
minimum concave cost flow setting, Ahmed et al.
\cite{AhmedEtAl2016} obtain an $O(T^{4R+7})$ bound for the
corresponding two-row grid with $R$ distinct production capacity
values.

Grid Theory provides a more direct additive interpretation in the
linear cost setting. The $R$ capacity values generate $R$ count
coordinates, yielding $O(T^{R+1})$ cumulative states and therefore an
$O(T^{R+2})$ GDP algorithm. Under the generic
concave cost recurrence, the same grid gives
$O(T^{2R+3})$, matching the exponent obtained in the general
Li--Li framework.
\subsection{Fixed-degree polynomial capacities}
\label{subsec:polynomial-capacities}

Suppose now that the capacity in period $t$ is given by
\begin{equation}
C_t
=
\sum_{\ell=0}^{d}
\gamma_\ell t^\ell,
\qquad
t\in\mathcal{T},
\label{eq:polynomial capacity}
\end{equation}
where the degree $d$ is fixed independently of $T$, and the
coefficients are such that $C_t\geq0$ over the planning horizon.

For every subset $S\subseteq\mathcal{T}$,
\begin{equation*}
\sum_{t\in S}C_t
=
\sum_{\ell=0}^{d}
\gamma_\ell
\left(
\sum_{t\in S}t^\ell
\right). 
\end{equation*}
Define the power sum coordinates $B_\ell(S)
:=
\sum_{t\in S}t^\ell$ with $\ell=0,\ldots,d$. Since $|S|\leq T$ and $t\leq T$, then $0\leq B_\ell(S)\leq T^{\ell+1}$.  Thus, the capacity profile admits an additive representation with $p=d+1$ coordinates. In particular, $\operatorname{adim}(\mathcal{B}) \leq d+1$. For fixed $d$, these coordinates are polynomial-time computable and have polynomially bounded encoding length, so the representation satisfies Definition \ref{def:additive-representation}.

The number of aggregate coordinate vectors is bounded by
\begin{equation*}
O
\left(
\prod_{\ell=0}^{d}T^{\ell+1}
\right)
=
O
\left(
T^{
\frac{(d+1)(d+2)}{2}
}
\right). 
\end{equation*}
After adding the cumulative demand anchors, the common grid envelope
satisfies
\begin{equation*}
|H^{\mathrm{poly}}|
=
O
\left(
T^{
1+\frac{(d+1)(d+2)}{2}
}
\right). 
\end{equation*}

\begin{corollary} 
\label{cor:polynomial capacity}
For every fixed degree $d$, the capacitated lot-sizing problem with capacity profile~\eqref{eq:polynomial capacity} is solvable by the GDP algorithm in $O
\left(
T^{
2+\frac{(d+1)(d+2)}{2}
}
\right)$  time.
\end{corollary}

For example, an affine capacity profile $C_t=\gamma_0+\gamma_1t$ has additive dimension at most two. Its subset sums are determined by the two coordinates $|S|$ and $\sum_{t\in S}t$.

\subsection{Seasonal and periodic capacities}
\label{subsec:seasonal-capacities}

Seasonal capacity patterns arise naturally when production availability repeats according to shifts, days of the week, weeks, or other fixed operating cycles.

Suppose that there exists a fixed integer $q\geq1$ and fixed values $c_1,\ldots,c_q$ such that
\begin{equation}
C_t
=
c_{1+((t-1)\bmod q)},
\qquad
t\in\mathcal{T}.
\label{eq:periodic-capacity}
\end{equation}
The capacity profile contains at most $q$ distinct values. Since $q$ is fixed independently of $T$, the result is an immediate consequence of Corollary~\ref{cor:fixed-capacity-levels}.

\begin{corollary} 
\label{cor:seasonal-capacity}
For every fixed period $q$, the capacitated lot-sizing problem with capacity profile~\eqref{eq:periodic-capacity} is solvable in polynomial time by the GDP algorithm.
\end{corollary}

If the $q$ positions of the seasonal cycle contain $R\leq q$ distinct capacity values, the common grid can be chosen with cardinality $O(T^{R+1})$,  and the resulting running time is $O(T^{R+2})$.

\subsection{Piecewise polynomial capacities}
\label{subsec:piecewise polynomial-capacities}

The preceding polynomial capacity result can be extended to profiles that follow different polynomial laws over different parts of the planning horizon.

Let $\mathcal{T}_1,\ldots,\mathcal{T}_R$ form a partition of $\mathcal{T}$ into $R$ consecutive pieces, where $R$ is fixed independently of $T$. Suppose that, for every
$t\in\mathcal{T}_r$,
\begin{equation}
C_t
=
\sum_{\ell=0}^{d}
\gamma_{r,\ell}t^\ell,
\qquad
r=1,\ldots,R,
\label{eq:piecewise polynomial capacity}
\end{equation}
where the degree $d$ is also fixed.

For every subset $S\subseteq\mathcal{T}$, define $B_{r,\ell}(S)
:=
\sum_{t\in S\cap\mathcal{T}_r}
t^\ell$ for $r=1,\ldots,R$ and  $\ell=0,\ldots,d$. Then
\begin{equation*}
\sum_{t\in S}C_t
=
\sum_{r=1}^{R}
\sum_{\ell=0}^{d}
\gamma_{r,\ell}B_{r,\ell}(S), 
\end{equation*}
and each coordinate satisfies $0\leq B_{r,\ell}(S)\leq T^{\ell+1}$. Since $R$ and $d$ are fixed, these coordinates are polynomial-time computable and have polynomially bounded encoding length. The profile therefore admits an additive representation with at most $R(d+1)$ coordinates, and $\operatorname{adim}(\mathcal{B})
\leq
R(d+1)$.  

The number of aggregate coordinate vectors is bounded by $O
\left(
T^{
R\frac{(d+1)(d+2)}{2}
}
\right)$. Consequently, the common cumulative production grid satisfies
\begin{equation*}
|H^{\mathrm{pw}}|
=
O
\left(
T^{
1+
R\frac{(d+1)(d+2)}{2}
}
\right). 
\end{equation*}

\begin{corollary} 
\label{cor:piecewise polynomial}
If both the number $R$ of pieces and the maximum polynomial degree
$d$ are fixed independently of $T$, then the capacitated lot-sizing
problem with capacity profile~\eqref{eq:piecewise polynomial capacity}
is solvable by the GDP algorithm in $O
\left(
T^{
2+
R\frac{(d+1)(d+2)}{2}
}
\right)$ time.
\end{corollary}

This family includes, for example, capacity profiles with successive
production phases, seasonal regimes, planned expansions, temporary
reductions, or different polynomial deterioration patterns over
different parts of the horizon.

\begin{table}[H]
\centering
\scriptsize
\renewcommand{\arraystretch}{1.25}
\begin{tabularx}{\textwidth}{
>{\raggedright\arraybackslash}p{0.15\textwidth}
>{\raggedright\arraybackslash}p{0.31\textwidth}
>{\raggedright\arraybackslash}p{0.24\textwidth}
>{\raggedright\arraybackslash}X}
\toprule
Family
&
Representative previous results
&
Grid Theory
&
Interpretation
\\
\midrule

Uncapacitated
&
Wagner and Whitin \cite{WagnerWhitin1958}:
$O(T^2)$ dynamic programming.
&
$|H|=O(T)$ and $O(T^2)$ GDP.
&
Recovers the classical quadratic order through cumulative demand
anchors.
\\\midrule

Arbitrary capacities
&
General cases are NP-hard
\cite{FlorianLenstraRinnooyKan1980,BitranYanasse1982}.
K{\i}rca \cite{Kirca1990} constructs exact cumulative production
state sets, whose worst-case size need not be polynomial.
&
The MGT gives an exact additive grid, which may be exponential.
&
The MGT is an optimality theorem; a separate compression
condition is needed for polynomiality.
\\\midrule

Constant capacity
&
Florian and Klein \cite{FlorianKlein1971}: $O(T^4)$ under
concave costs.
Van Hoesel and Wagelmans \cite{vanHoeselWagelmans1996}:
$O(T^3)$ with concave production and linear holding costs.
&
Linear costs:
$|H|=O(T^2)$,
$O(T^3)$ GDP.

Concave costs:
generic $O(T^5)$.
&
Recovers the classical cubic order in the linear cost model through
one additive coordinate.
\\\midrule

Constant MOQ
&
Okhrin and Richter \cite{OkhrinRichter2011b}:
$O(T^2)$ for a special linear uncapacitated MOQ model.
&
Linear costs:
$|H|=O(T^2)$,
$O(T^3)$ GDP.
&
Polynomiality is known; Grid Theory identifies the lower boundary as
the resource generating the cumulative state space.
\\\midrule

Constant capacity
+ constant MOQ
&
Okhrin and Richter \cite{OkhrinRichter2011a}: $O(T^3)$ under
restricted costs.

Hellion et al.
\cite{HellionMangionePenz2012,HellionMangionePenz2013}:
$O(T^6)$ under concave costs.

Koca et al. \cite{KocaYamanAkturk2014}: $O(T^6)$.

Li and Li \cite{LiLi2016}: generic $O(T^7)$.
&
Linear costs:
$|H|=O(T^3)$,
$O(T^4)$ GDP.

Concave costs:
generic $O(T^7)$.
&
Not a new polynomial family. Grid Theory gives a common
two-coordinate explanation through $kL+\ell C$.
\\\midrule

$R$ capacity levels
&
Li and Li \cite{LiLi2016}:
$O(T^{2R+3})$.

Ahmed et al. \cite{AhmedEtAl2016}:
$O(T^{4R+7})$ in the broader two-row
minimum concave cost flow model.
&
Linear costs:
$|H|=O(T^{R+1})$,
$O(T^{R+2})$ GDP.

Concave costs:
$O(T^{2R+3})$.
&
The capacity levels become additive count coordinates. The
linear cost GDP gives a sharper exponent than the generic
frameworks.
\\\midrule

Nondecreasing capacities
with structured costs
&
Chung and Lin \cite{ChungLin1988}:
$O(T^2)$ under nonincreasing setup and unit production costs.
&
No polynomial grid guarantee for an arbitrary nondecreasing
capacity sequence.
&
An important example of polynomiality arising through cost
dominance rather than bounded additive complexity.
\\\midrule

Fixed-degree polynomial
capacities
&
No direct counterpart identified in the lot-sizing literature reviewed for this paper.
&
$\displaystyle
|H|=
O\!\left(
T^{1+\frac{(d+1)(d+2)}{2}}
\right)$,

$\displaystyle
O\!\left(
T^{2+\frac{(d+1)(d+2)}{2}}
\right)$
GDP.
&
The number of distinct capacities may grow with $T$, but their
subset sums have fixed-dimensional power sum representations.
\\\midrule

Piecewise polynomial
capacities
&
No direct counterpart identified in the literature reviewed here.
&
$\displaystyle
O\!\left(
T^{2+
R\frac{(d+1)(d+2)}{2}}
\right)$
GDP for fixed $R,d$.
&
Extends additive compression beyond fixed finite collections of
capacity values.
\\

\bottomrule
\end{tabularx}
\caption{Representative complexity results and their relation to Grid Theory.}
\label{tab:literature-comparison}
\end{table}

We emphasize that piecewise polynomial capacities refer here to
the dependence of the exogenous capacity $C_t$ on the time index $t$.
This should not be confused with the piecewise concave
\emph{production cost} functions studied, for example, by Koca,
Yaman, and Akturk \cite{KocaYamanAkturk2014}.

\subsection{Relation to previous complexity results}
\label{subsec:literature-comparison}

The preceding results include both classical polynomial families and
new consequences of the additive grid framework. It is therefore
useful to distinguish between recovering polynomiality through a
common structural argument and improving the best specialized
algorithm for a particular model. Table~\ref{tab:literature-comparison}
summarizes representative complexity results from the literature and
their relation to Grid Theory.

Table~\ref{tab:literature-comparison} also emphasizes that the
polynomial grid condition is sufficient rather than necessary.
The Chung--Lin family is particularly instructive: polynomiality can
result from cost-induced dominance even when the additive
capacity generated grid itself is not guaranteed to have polynomial
cardinality.

\section{Discussion}
\label{sec:discussion}

The results of this paper concern  dynamic lot-sizing with exogenous lower and upper production bounds, but their interpretation is broader than capacity alone. The structural mechanism identified by the Grid Theory rests on three ingredients: regeneration anchors, a boundary structure of extreme solutions, and additive compression of the boundary values that generate cumulative production states.

\subsection{What the framework explains}
\label{sec:explains}

The polynomial families of Section \ref{sec:polynomial-families} show that what matters is not merely the number of periods or the number of distinct boundary values, but the complexity of their additive combinations. This distinction is particularly visible for fixed-degree polynomial capacities. Bounded additive dimension is strictly more flexible than a description based only on a fixed finite set of capacity levels.

The proposed framework also separates two questions that are often merged. The Main Grid Theorem is an optimality-preserving discretization result, while the General Polynomiality Theorem is a compression result. The exponential grid example of Section \ref{subsec:exponential-grid} shows why the distinction is necessary.

\subsection{What the framework does not imply}
\label{sec:limitations}

Bounded additive dimension is a sufficient condition for polynomial solvability, not a characterization of all polynomially solvable lot-sizing families. An exact additive grid may be exponential even though a dominance argument or additional cost structure may yield a polynomial algorithm. Also, exponential cardinality of the additive grid does not imply NP-Hardness. It only shows that explicit enumeration of the grid, by itself, is not a polynomial solution method.

The present results also rely mainly on the boundary values being exogenous input data. If the admissible production boundaries depend on additional operational decisions, the grid can no longer necessarily be constructed independently of the state trajectory. For instance, dynamic capacity adjustment is an example of this situation. In such models, polynomiality would depend jointly on the complexity of the state-transition system generating the boundaries and on the additive complexity of the boundary values themselves.

For the same reason, additional resources such as energy consumption, carbon emissions, or budgets are not covered simply because they are additive. If such resources introduce new state variables or coupling constraints, a new structural theorem is required before a grid argument can be applied. In particular, one would need an analogue of the Main Grid Theorem in the enlarged state space, not just a low-dimensional representation of the resource coefficients.

\subsection{Directions suggested by the framework}
\label{sec:directions}

The present theory suggests three natural extensions.

First, one may study \textit{multidimensional grid representations}. If production decisions generate several cumulative resources and each admits a bounded-dimensional additive representation, then the combined resource vector may still have polynomially many aggregate descriptions when the number of resources is fixed. This observation alone is not sufficient for a complexity result, but it identifies the relevant compression mechanism that a multidimensional MGT would have to exploit.

Second, one may consider \textit{endogenous boundary systems}. If an operational state selects among several production limits, the appropriate question is whether the state process itself has a polynomially representable transition structure and whether the boundary values generated along that process retain low additive complexity. This would extend Grid Theory from exogenous resource profiles to resource profiles generated dynamically by operational decisions.

Last, one may investigate compression mechanisms beyond additive dimension. The present framework provides a sufficient condition. Other algebraic relations, dominance properties, cost-induced pruning rules, or alternative state representations may produce polynomial optimality preserving sets even when the complete additive envelope is large. Understanding how such mechanisms interact with the MGT could substantially enlarge the class of problems explained by the framework.

These extensions have to be viewed as directions for further theory rather than consequences of the results developed here. The contribution of the present paper is more precise: for exogenous production boundaries, additive structure provides a general explanation of when the cumulative production state space induced by the MGT is polynomially representable.

\section{Conclusion}
\label{sec:conclusion}

This paper develops Grid Theory as a structural framework for dynamic lot-sizing with lower and upper production boundaries. Its starting point is the observation that the geometry of optimal extreme solutions restrict the production quantities that need to be considered inside regeneration  intervals.

The first result establishes that every regeneration interval of a relevant extreme solution contains at most one partially utilized production period. The Main Grid Theorem then translates this local boundary structure into a global cumulative production statement: there exists an optimal trajectory supported on forward and backward additive grids generated from cumulative-demand anchors and production-boundary values.

This finite-state representation alone does not imply polynomiality. The exact additive grid may be exponentially large. The central complexity contribution is therefore the distinction between \textit{grid existence} and \textit{grid compressibility}. The notion of additive dimension captures one sufficient compression mechanism. When the boundary resource profile admits a bounded-dimensional, polynomially accessible additive representation, the General Polynomiality Theorem yields a polynomial-size grid envelope and hence an exact polynomial-time GDP algorithm.

The framework recovers several classical polynomial families and places them under a common structural explanation, including constant capacities, constant minimum order quantities, and a fixed number of capacity levels. More importantly, it also applies to profiles whose number of distinct capacity values grows with the horizon, such as fixed-degree polynomial and fixed-piecewise-polynomial capacity sequences. These examples illustrate that tractability is managed not by the number of distinct resource values itself, but by the complexity of their additive aggregate representation.

Grid Theory is not to be viewed as a necessary characterization of polynomial lot-sizing problems, nor as a replacement for specialized algorithms that exploit stronger model-specific structure. Its purpose is to isolate a reusable mechanism: some continuous dynamic lot-sizing models are polynomial because at least one optimal cumulative production trajectory can be represented on a polynomially constructible additive state space.

\bibliographystyle{plainurl} 
 
\bibliography{biblio}

\end{document}